\documentclass{amsproc}
\usepackage{amsmath}
\usepackage{amssymb}
\input xy
\xyoption{all}

\newtheorem{theorem}{Theorem}[section]

\theoremstyle{definition}

\newtheorem{example}[theorem]{Example}
\newtheorem{Lemma}[theorem]{Lemma}
\newtheorem{corollary}[theorem]{Corollary}
\newtheorem{Prop}[theorem]{Proposition}
\theoremstyle{remark}
\newtheorem{remark}[theorem]{Remark}

\numberwithin{equation}{section}
\newcommand{\Cal}[1]{{\mathcal #1}}
\newcommand{\codim}{\mbox{\rm codim}}
\newcommand{\End}{\operatorname{End}}
\newcommand{\Hom}{\operatorname{Hom}}
\DeclareMathOperator{\Ob}{Ob}
\newcommand{\Mod}{\operatorname{Mod-\!}}
\newcommand{\add}{\mbox{\rm add}}
\newcommand{\Z}{\mathbb{Z}}
\newcommand{\Q}{\mathbb{Q}}
\newcommand{\cmat}{\left(\begin{array}}
\newcommand{\fmat}{\end{array}\right)}
\newcommand{\soc}{\mbox{\rm soc}}

\begin{document}

 \title[On a category of extensions with at most four maximal ideals]{On a category of extensions whose endomorphism rings have at most four maximal ideals}
  

 \author{Federico Campanini}
\address{Dipartimento di Matematica, Universit\`a di Padova, 35121 Padova, Italy}
\email{federico.campanini@math.unipd.it, facchini@math.unipd.it}
 
\author{Alberto Facchini}
\thanks{Partially supported by Dipartimento di Matematica ``Tullio Levi-Civita'' of Universit\`a di Padova (Project BIRD163492/16 ``Categorical homological methods in the study of algebraic structures'' and Research program DOR1690814 ``Anelli e categorie di moduli'').}



\begin{abstract} We describe the endomorphism ring of a short exact sequences $\xymatrix@1{ 0 \ar[r] &A_R \ar[r]&B_R \ar[r] &C_R \ar[r] &0}$ with $A_R$ and $C_R$ uniserial modules and the behavior of these short exact sequences as far as their direct sums are concerned.
\end{abstract}

\maketitle

\section{Introduction}

Two right modules $A_R$ and $B_R$ are said to have: (1) the same  {\em monogeny class}, denoted
by $[A_R]_m=[B_R]_m$, if there exist a monomorphism $A_R\rightarrow B_R$ and
a monomorphism $B_R\rightarrow A_R$; (2)  the same {\em epigeny class}, denoted
by $[A_R]_e=[B_R]_e$, if there exist an epimorphism $A_R\rightarrow B_R$ and
an epimorphism $B_R\rightarrow A_R$. A module $A_R$ is {\em uniserial} if  the lattice of its
submodules is linearly ordered under inclusion.
That is, if,
for any submodules $V$ and $W$ of $A_R$, either $V\subseteq W$ or $W\subseteq V$. Finite direct sums of uniserial modules are classified via their monogeny class and their epigeny class \cite{TAMS}. These two invariants are sufficient because the endomorphism ring of a uniserial module has at most two maximal ideals. Other classes of modules whose endomorphism rings have at most two maximal ideals have similar behaviours \cite{AAF1, Tufan, FG}.

In this article, we focus our attention to short exact sequences $$\xymatrix@1{ 0 \ar[r] &A_R \ar[r]&B_R \ar[r] &C_R \ar[r] &0}$$ of right $R$-modules. They form the class of objects of a category $\Cal E$, which is additive, but is not abelian. If we consider the objects of $\Cal E$ for which $A_R$ and $C_R$ are uniserial $R$-modules, we find a behaviour very similar to that of uniserial modules mentioned in the previous paragraph. In this case, we have that finite direct sums of short exact sequences with $A_R$ and $C_R$ uniserial modules are classified via four invariants, which are the natural generalizations of monogeny class and epigeny class. Four invariants are necessary now because the endomorphism ring in $\Cal E$ of a short exact sequence with $A_R$ and $C_R$ uniserial modules has at most four maximal ideals. 

The main results in this paper are Theorem~\ref{EARY}, in which we describe the endomorphism ring of a short exact sequence $\xymatrix@1{ 0 \ar[r] &A_R \ar[r]&B_R \ar[r] &C_R \ar[r] &0}$ with $A_R$ and $C_R$ uniserial modules, and Theorem~\ref{completo'}, in which we describe the behaviour of the short exact sequences with $A_R$ and $C_R$ uniserial  as far as their direct-sum decompositions in $\Cal E$ are concerned.

In this article, modules are right unitary modules. 

\section{Modules with a semilocal endomorphism ring}

Let $R$ be a ring, associative and with an identity $1$. We consider the following category $\Cal E$ of extensions. The objects of $\Cal E$ are the short exact sequences $$\xymatrix@1{ 0 \ar[r] &A_R \ar[r]^\alpha &B_R \ar[r]^\beta &C_R \ar[r] &0, }$$where $A_R,B_R,C_R$ are right $R$-modules. A morphism in $\Cal E$ between two such exact sequences
\begin{equation}\label{1}
\xymatrix@1{ 0 \ar[r] &A_R \ar[r]^\alpha &B_R \ar[r]^\beta &C_R \ar[r] &0}
\end{equation}
and
\begin{equation}\label{2}
\xymatrix@1{ 0 \ar[r] &A'_R \ar[r]^{\alpha'} &B'_R \ar[r]^{\beta'} &C'_R \ar[r] &0}
\end{equation}
is a right $R$-module morphism $f\colon B_R\to B_R'$ that induces a commutative diagram 
\begin{equation*}
  \label{D}
  \xymatrix{
 0\ar[r]&   A_R \ar[r]^\alpha \ar[d]_{f|^A_{A'}} &B_R \ar[r]^\beta \ar[d]_{f} &C_R \ar[r] \ar[d]^{\overline{f}} &0 \\
0\ar[r]&   A_R' \ar[r]_{\alpha'} &B'_R \ar[r]_{\beta'} &C'_R \ar[r] &0.
  }
\end{equation*}
Equivalently, $f\colon B_R\to B_R'$ is a right $R$-module morphism such that $f(\alpha(A_R))\subseteq \alpha'(A_R')$. We will denote by $E_B$ the endomorphism ring of the object
\begin{equation}\label{2VLH}\xymatrix@1{ 0 \ar[r] &A_R \ar[r]^\alpha &B_R \ar[r]^\beta &C_R \ar[r] &0}
\end{equation}
in the category $\Cal E$. Notice that in this notation $E_B$ for the endomorphism ring of the exact sequence (\ref{2VLH}), there is a slight abuse of notation, because we assume that the modules $A_R$ and $C_R$ and the morphisms $\alpha$ and $\beta$ are fixed once for all, which will be true in the rest of this article.

Clearly, the zero object of the category $\Cal E$ is the exact sequence
$$
\xymatrix@1{ 0 \ar[r] &0 \ar[r] &0 \ar[r] &0 \ar[r] &0}.
$$
It will be denoted by $0$.

If (\ref{1}) and (\ref{2}) are two objects of $\Cal E$ and $f\colon B_R\to B_R'$ is a morphism in $\Cal E$, that is, a right $R$-module morphism with $f(\alpha(A_R))\subseteq \alpha'(A_R')$, then $f$ is an isomorphism in $\Cal E$ if and only if $f\colon B_R\to B_R'$ is a right $R$-module isomorphism and $f(\alpha(A_R))= \alpha'(A_R')$. Also, the category $\Cal E$ is not an abelian category. In fact, it is easy to construct  a morphism $f$ in $\Cal E$ for which $f\colon B_R\to B_R'$ is a right $R$-module isomorphism and $f(\alpha(A_R))\subset \alpha'(A_R')$. Such a morphism  $f$ in $\Cal E$ is clearly both right cancellable and left cancellable, that is, is both monic and epi, but is not an isomorphism in $\Cal E$. Thus $\Cal E$  is not abelian.

\medskip

A ring $R$ is {\em semilocal} if the factor ring $R/J(R)$, where $J(R)$ denotes the Jacobson radical of $R$, is a semisimple artinian ring. A ring $R$ is semilocal if and only if its dual dimension $\codim(R)$ is finite \cite[Proposition~2.43]{libro}. If $R$ and $S$ are rings, a ring morphism $\varphi \colon R \to S$
is  {\em local} if, for every $r\in R$, $\varphi (r)$ invertible in $S$ implies $r$ invertible in $R$ \cite{CD}. 

\begin{Prop}\label{localand codim} \cite[Corollary 2]{CD}
If $R\to S$ is a local morphism between two rings $R$ and $S$, then $\codim(R)\le \codim(S)$.
\end{Prop}

Recall that a ring $S$ {\em has type $n$} \cite{AlbPav4}, if the factor ring $S/J(S)$ is a direct product of $n$ division rings, and that $S$ is a {\em ring of finite type} if it has type $n$ for some integer $n\ge 1$. If a ring $S$ has finite type, then the type $n$ of $S$ coincides with the dual Goldie dimension $\codim(S)$ of $S$ \cite[Proposition~2.43]{libro}. A module $M_R$ has {\em type $n$} ({\em finite type}) if its endomorphism ring $\End(M_R)$ is a ring of type $n$ (of finite type, respectively). The zero module is the unique module of type $0$. 

\begin{Prop}\label{Type} \cite[Proposition~2.1]{AlbPav4}
Let $S$ be a ring and $n\ge1$ be an integer. The following conditions are equivalent:
\begin{enumerate}
\item[{\rm (a)}]  The ring $S$ has type $n$.
\item[{\rm (b)}]  $n$ is the smallest of the positive integers $m$ for which there is a local morphism of $S$ into a direct product of $m$ division rings.
\item[{\rm (c)}]  The ring $S$ has exactly $n$ distinct maximal right ideals, and they are all two-sided ideals in $S$.
\item[{\rm (d)}]  The ring $S$ has exactly $n$ distinct maximal left ideals, and they are all two-sided ideals in $S$.
\end{enumerate}
\end{Prop}

\begin{Prop}\label{semilocal}
Let $\xymatrix@1{ 0 \ar[r] &A_R \ar[r]^\alpha &B_R \ar[r]^\beta &C_R \ar[r] &0 }$ be a short exact sequence of right $R$-modules. Then:
\begin{enumerate}
\item[{\rm (a)}]   $\codim(E_B)\le\codim(\End(A_R))+\codim(\End(C_R))$. 
\item[{\rm (b)}]   If $A_R$ and $C_R$ are modules whose endomorphism ring is semilocal, then the endomorphism ring $E_B$ of the short exact sequence in the category $\Cal E$ is also a semilocal ring.
\item[{\rm (c)}]    If $A_R$ and $C_R$ are modules of type $n$ and $m$ respectively, then the endomorphism ring $E_B$ of the short exact sequence is of type $\le n+m$.
\end{enumerate}
\end{Prop}

\begin{proof} Without loss of generality, we can suppose that $A_R$ is a submodule of $B_R$, $\alpha\colon A_R\to B_R$ is the embedding, $C_R=B_R/A_R$ and $\beta$ is the canonical projection. There is a ring morphism $\varphi\colon E_B\to\End(A_R)\times\End(B_R/A_R)$ defined by $f\mapsto (f|^A_{A},\overline{f})$, which is a local morphism by the Snake Lemma. Now (a) follows from Proposition~\ref{localand codim}, and (b) follows from the fact that a ring is semilocal if and only if its dual Goldie dimension is finite \cite[Proposition~2.43]{libro}.

As far as (c) is concerned, suppose that $A_R$ and $C_R$ are modules of type $n$ and $m$ respectively. By  Proposition~\ref{Type}(b), there are local morphisms  $\End(A_R)\to D_1\times\dots\times D_n$ and $\End(B_R/A_R)\to D'_1\times\dots\times D'_m$ for suitable division rings $D_1,\dots, D_n,D'_1,\dots, D'_m$. Composing with $\varphi$, we get a local morphism $E_B\to D_1\times\dots\times D_n\times  D'_1\times\dots\times D'_m$. Thus $E_B$ has type $\le n+m$ by Proposition~\ref{Type}(b).\end{proof}

\begin{remark}\label{rem'} In the particular case where either $A_R$ is zero or $C_R$ is zero, their endomorphism ring is the trivial ring with one element. If we define such a ring to be of dual Goldie dimension $0$ and we consider it to be a semilocal ring of finite type $0$, then Proposition~\ref{semilocal} also holds in this case.
\end{remark}

\section{Uniserial modules}\label{UM}

The endomorphism ring $E:=\End(U_R)$ of a non-zero uniserial module $U_R$ \cite[Theorem~1.2]{TAMS} has at most two maximal right ideals: the two-sided completely prime ideals  $I_U\colon=\{\,f\in E \mid f$ is not  injective$\,\}$ and
$K_U\colon=\{\,f\in E \mid f$ is not  surjective$\,\}$, or only one of them. Here, a {\em completely prime ideal} $P$ of a ring $S$ is a proper ideal $P$ of~$S$ such that, for every $x,y\in S$, $xy\in P$ implies that either $x\in P$ or $y\in P$. 

\begin{remark}\label{rem} If $S$ is a ring, $P_1,\dots, P_n$ are completely prime two-sided ideals of $S$ and $L$ is a right ideal of $S$ contained in $\bigcup_{i=1}^n P_i$, then $L\subseteq P_i$ for some $i=1,2,\dots,n$. 
\end{remark}

\begin{theorem}\label{EARY} Let $A< B_R$  be right $R$-modules with $A_R\ne 0$. Suppose $A_R$ and $C_R\colon=B_R/A$ uniserial modules. 
Let $E_B$ be the endomorphism ring of the short exact sequence $$\xymatrix@1{ 0 \ar[r] &A_R \ar[r] &B_R \ar[r] &B_R/A \ar[r] &0 },$$ viewed as an object of the category $\Cal E$.  
Set 
\begin{align*} I_{B,m,l}&:=\{\,f\in E_B \mid \ker(f)\cap A\ne 0\,\},\\ I_{B,e,l}&:=\{\,f\in E_B \mid f(A)\subset A\,\},\\
I_{B,m,u}&:=\{\,f\in E_B \mid A\subset f^{-1}(A)\,\}\end{align*} and 
$$I_{B,e,u}:=\{\,f\in E_B \mid f(B)+A\subset B\,\}.$$
Then $I_{B,m,l},I_{B,e,l},I_{B,m,u}$ and $I_{B,e,u}$ are four two-sided completely prime ideals of $E_B$, and
every proper right ideal of $E_B$ and every proper left ideal of $E_B$ is contained in one of these four ideals of $E_B$.
Moreover, $E_B/J(E_B)$ is isomorphic to the direct product of $n$ division rings $D_1,\dots,D_n$ for some $n$ with $1\le n\le 4$, where $\{D_1,\dots,D_n\}\subseteq \{E_B/I_{B,m,l},E_B/I_{B,e,l},E_B/I_{B,m,u},E_B/I_{B,e,u}\}$.
\end{theorem}

\begin{proof} As in the proof of Proposition~\ref{semilocal}, there is a canonical local morphism $\varphi\colon E_B\to\End(A_R)\times\End(B_R/A_R)$. Since $A_R$ is a uniserial module, the canonical projections define a local morphism $\End(A_R)\to \End(A_R)/I_A\times \End(A_R)/K_A$, where $\End(A_R)/I_A$ and $\End(A_R)/K_A$ are integral domains. More precisely, three cases can occur:

(1) $I_A\subseteq K_A$. In this case, $\End(A_R)$ is a local ring with maximal ideal $K_A$, $\End(A_R)/K_A$ is a division ring, and the canonical projection $$\End(A_R)\to\End(A_R)/ K_A$$ is a local morphism. 

(2) $K_A\subseteq I_A$. In this case, $\End(A_R)$ is a local ring with maximal ideal $I_A$, $\End(A_R)/I_A$ is a division ring, and the canonical projection $$\End(A_R)\to \End(A_R)/I_A$$ is a local morphism. 

(3) $I_A\nsubseteq K_A$ and $K_A\nsubseteq I_A$. In this case, $\End(A_R)$ has two maximal right ideal $I_A$ and $K_A$, which are two-sided ideals, $\End(A_R)/I_A$ and $\End(A_R)/K_A$ are two division rings, and the canonical mapping $$\End(A_R)\to \End(A_R)/I_A\times \End(A_R)/K_A$$ is a local morphism. 

Similarly for the uniserial module $C_R:= B_R/A$. Thus if $D_1,\dots,D_n$ are those $n$ rings between the rings $\End(A_R)/I_A$, $\End(A_R)/K_A$, $\End(C_R)/I_C$ and $\End(C_R)/K_C$ that are division rings, we get a local morphism $E_B\to D_1\times\dots\times D_n$ of $E_B$ into $n\le 4$ division rings. The four ideals $I_{B,m,l},I_{B,e,l},I_{B,m,u}$ and $I_{B,e,u}$ of $E_B$ are the kernels of the canonical morphisms $E_B\to \End(A_R)/I_A$, $E_B\to \End(A_R)/K_A$, $E_B\to \End(C_R)/I_C$, $E_B\to \End(C_R)/K_C$, respectively. It is therefore immediate that they are completely prime two-sided ideals of $E_B$.

The non-invertible elements of $\End(A_R)$ are exactly the elements of $I_A\cup K_A$. Similarly, the non-invertible elements of $\End(C_R)$ are exactly the elements of $I_C\cup K_C$. Since $\varphi\colon E_B\to\End(A_R)\times\End(B_R/A_R)$ is a local morphism, the non-invertible elements of $E_B$ are exactly the elements of $I_{B,m,l}\cup I_{B,e,l}\cup I_{B,m,u}\cup I_{B,e,u}$. But these four ideals of $E_B$ are completely prime, so that every proper right (or left) ideal of $E_B$, which necessarily consists of non-invertible elements, must be contained in one of them (Remark~\ref{rem}). In particular, every maximal two-sided ideal of $E_B$ must be one of $I_{B,m,l},I_{B,e,l},I_{B,m,u}$ or $I_{B,e,u}$.

The last assertion follows by applying the Chinese Remainder Theorem.
\end{proof}

\begin{remark} Note that in Theorem~\ref{EARY}, without the assumption $0\neq A_R< B_R$, the sets $I_{B,m,l},\ I_{B,e,l},\ I_{B,m,u}$ and $I_{B,e,u}$ could be empty. For example, for $A_R=0$, we get that $I_{B,m,l}$ and $I_{B,e,l}$ are empty sets, while $I_{B,m,u}$ and $I_{B,e,u}$ coincide with the set of $R$-module morphisms of $B_R$ that are not injective or not surjective respectively. Similarly for $C_R=0$. Anyway, if $A$ or $C$ is equal to zero, $B$ turns out to be a uniserial module and the corresponding description of the endomorphism ring $E_B$ of $B$ is well known \cite[Theorem~1.2]{TAMS}. Of course, these are the only two cases in which one of the four sets is empty, because otherwise the zero morphism belongs to $I_{B,m,l},I_{B,e,l},I_{B,m,u}$ and $I_{B,e,u}$. Indeed, in the hypothesis of Theorem~\ref{EARY}, these four sets are always proper ideals of $E_B$.
\end{remark}

Let (\ref{1}) and (\ref{2}) be two objects of $\Cal E$. We will write:

\noindent
$[B]_{m,l}=[B']_{m,l}$ if there exist two morphisms $f\colon B_R\to B'_R$ and $g\colon B'_R\to B_R$ in the category $\Cal E$ such that $f|^A_{A'}$ and $g|^{A'}_{A}$ are injective right $R$-module morphisms.

\noindent
$[B]_{e,l}=[B']_{e,l}$ if there exist two morphisms $f\colon B_R\to B'_R$  and $g\colon B'_R\to B_R$  in the category $\Cal E$ such that $f(A)=A'$ and $g(A')=A$.

\noindent
$[B]_{m,u}=[B']_{m,u}$ if there exist two morphisms $f\colon B_R\to B'_R$  and $g\colon B'_R\to B_R$  in the category $\Cal E$ such that $\overline{f}\colon C\to C'$ and $\overline{g}\colon C'\to C$ are injective right $R$-module morphisms.

\noindent
$[B]_{e,u}=[B']_{e,u}$ if there exist two morphisms $f\colon B_R\to B'_R$  and $g\colon B'_R\to B_R$  in the category $\Cal E$ such that $\overline{f}\colon C\to C'$ and $\overline{g}\colon C'\to C$ are surjective right $R$-module morphisms.

\noindent
(Here the four letters $m,e,l,u$ stand for ``monogeny'', ``epigeny'', ``lower'' and ``upper'', respectively.)

If $[B]_{a,b}=[B']_{a,b}$ for some $a=m,e$ and $b=l,u$, we will say that the two objects (\ref{1}) and (\ref{2}) have the same {\em $(a,b)$-class}.

In this notation, $[B]_{m,l}=[0]_{m,l}$ if and only if $[B]_{e,l}=[0]_{e,l}$, if and only if $A_R$ is the zero module. Similarly, $[B]_{m,u}=[0]_{m,u}$ if and only if $[B]_{e,u}=[0]_{e,u}$, if and only if $C_R$ is the zero module.

\section{Isomorphism}

In this section, we will consider the full subcategory $\Cal U$ of the category $\Cal E$ whose elements are all short exact sequences $\xymatrix@1{ 0 \ar[r] &A_R \ar[r] &B_R \ar[r] &C_R \ar[r] &0 }$ with $A_R, C_R$ non-zero uniserial right $R$-modules. In particular, every element of $\Cal U$ has a semilocal endomorphism ring (Proposition~\ref{semilocal}(b)). 
 
\begin{remark}\label{1.1} Let {\rm (\ref{1})} and {\rm (\ref{2})} be two objects of $\Cal U$. By \cite[Lemma~1.1]{TAMS}, $[B]_{a,b}=[B']_{a,b}$ if and only if there exist two morphisms $f\colon B\to B'$ and $g\colon B' \to B$ in the category $\Cal U$ such that $g f \notin I_{B,a,b}$ (or, equivalently, such that $f g \notin I_{B',a,b}$), where $a=m,e$ and $b=l,u$.\end{remark}
 
\begin{Lemma} \label{4.2} Let {\rm (\ref{1})} and {\rm (\ref{2})} be two objects in the category $\Cal U$. Fix $a=m,e$ and $b=l,u$. Suppose that $[B]_{a,b}=[B']_{a,b}$. Then the following properties hold:
\begin{enumerate}
\item[{\rm (a)}]
For every $c=m,e$ and $d=l,u$, one has that $I_{B,c,d}\subseteq I_{B,a,b}$ if and only if $I_{B',c,d}\subseteq I_{B',a,b}$.

\item[{\rm (b)}]
For every $c=m,e$ and $d=l,u$, $I_{B,c,d}\subseteq I_{B,a,b}$ implies $[B]_{c,d}=[B']_{c,d}$.

\end{enumerate}
\end{Lemma}

\begin{proof}
{\rm (a)}
It suffices to show that if $c=m,e$, $d=l,u$ and $I_{B,c,d}\subseteq I_{B,a,b}$, then $I_{B',c,d}\subseteq I_{B',a,b}$. Fix $c=m,e$ and $d=l,u$ and suppose that $I_{B',c,d}\nsubseteq  I_{B',a,b}$. Let $\varphi\colon B'_R\to B'_R$ be a morphism in $I_{B',c,d}$ not in $I_{B',a,b}$. Since $[B]_{a,b}=[B']_{a,b}$, there exist two morphisms $f\colon B\to B'$ and $g\colon B' \to B$ in the category $\Cal U$ such that $g f \notin I_{B,a,b}$. In particular, by \cite[Lemma~1.1]{TAMS}, $g\varphi f\in I_{B,c,d}$ and $g\varphi f \notin  I_{B,a,b}$, which implies that $I_{B,c,d}\nsubseteq  I_{B,a,b}$.

{\rm (b)} Suppose that $I_{B,c,d}\subseteq I_{B,a,b}$ for some $c=m,e$ and $d=l,u$. Since $[B]_{a,b}=[B']_{a,b}$, there exist two morphisms $f\colon B\to B'$ and $g\colon B' \to B$ in the category $\Cal U$ such that $g f \notin I_{B,a,b}$. In particular $g f \notin I_{B,c,d}$, and therefore $[B]_{c,d}=[B']_{c,d}$ by Remark \ref{1.1}.

\end{proof}

\begin{corollary}\label{4.3}
Let (\ref{1}) and (\ref{2}) be two objects of $\Cal U$ with $[B]_{a,b}=[B']_{a,b}$ for every $a=m,e$ and every $b=l,u$. Consider the sets $\Cal S_B\colon=\{I_{B,m,l},I_{B,e,l},I_{B,m,u},$ $I_{B,e,u}\}$ and $\Cal S_{B'}\colon=\{I_{B',m,l},I_{B',e,l},I_{B',m,u},I_{B',e,u}\}$ partially ordered by set inclusion. Then the canonical mapping $\Phi\colon \Cal S_B\to \Cal S_{B'}$, defined by 
$I_{B,a,b}\mapsto I_{B',a,b}$, is a partially ordered set isomorphism.
\end{corollary}

\begin{proof}
It immediately follows from Lemma \ref{4.2} {\rm (a)}.
\end{proof}

\begin{Prop}\label{propiso}
Let (\ref{1}) and (\ref{2}) be two objects of $\Cal U$. Then (\ref{1}) and (\ref{2}) are isomorphic in $\Cal U$ if and only if $[B]_{a,b}=[B']_{a,b}$ for every $a=m,e$ and $b=l,u$.
\end{Prop}

\begin{proof}
Assume that $[B]_{a,b}=[B']_{a,b}$ for every $a=m,e$ and $b=l,u$. By Theorem~\ref{EARY}, the maximal right ideals of $E_B$ are the ideals $I_{B,a,b}$, where $(a,b)$ ranges in a non-empty subset $S$ of $\{(m,l),(e,l),(m,u),(e,u) \}$. Consequently, the maximal right ideals of $E_{B'}$ are the ideals $I_{B',a,b}$ for the same pairs $(a,b)$ in $S$ (Corollary~\ref{4.3}).

By hypothesis, there are morphisms $f_{(a,b)}\colon B\rightarrow B'$ and $f'_{(a,b)}\colon B'\rightarrow B$ such that $f'_{(a,b)}f_{(a,b)}\notin I_{B,a,b}$ for every $(a,b)\in S$.
Moreover, since $E_B/J(E_B)$ and $\prod_{(a,b)\in S}E_B/I_{B,a,b}$ are canonically isomorphic, for every $(a,b)\in S$ there exists $\varepsilon_{(a,b)}\in E_B$ such that $\varepsilon_{(a,b)} \equiv \delta_{(a,b),(c,d)}\pmod{ I_{B,c,d}}$ for every $(c,d) \in S$. Similarly, for every $(a,b)\in S$ there exists $\varepsilon'_{(a,b)}\in E_{B'}$ such that $\varepsilon'_{(a,b)} \equiv \delta_{(a,b),(c,d)}$ (mod $I_{B',c,d}$) for every $(c,d) \in S$.
The two morphisms $\varphi\colon=\sum_{(a,b)\in S} \varepsilon'_{(a,b)} f_{(a,b)} \varepsilon_{(a,b)}$ and $\psi\colon=\sum_{(a,b)\in S} \varepsilon_{(a,b)} f'_{(a,b)} \varepsilon'_{(a,b)}$ are such that $\psi \varphi \in E_B$ is invertible modulo $J(E_B)$ and $\varphi \psi \in E_{B'}$ is invertible modulo $J(E_{B'})$. Hence $\psi \varphi$ and $\varphi \psi$ are invertible in $E_B$ and $E_{B'}$ respectively, and therefore $\varphi\colon B_R\rightarrow B'_R$ is an isomorphism.

The inverse implication is clear.
\end{proof}

Let $\Cal C$ be any preadditive category. An {\em ideal} of $\Cal C$ assigns to every pair $A,B$ of objects of $\Cal C$ a subgroup $\Cal I(A,B)$ of the abelian group $\Hom_{\Cal C}(A,B)$ with the property that, for all morphisms $\varphi\colon C\rightarrow A$, $\psi\colon A\rightarrow B$ and $\omega\colon B\rightarrow D$ with $\psi \in \Cal I(A,B)$, one has that $\omega\psi\varphi \in \Cal I(C,D)$ \cite[p.~18]{Several}.

A {\em completely prime ideal} $\Cal P$ of $\Cal C$ consists of a subgroup $\Cal P(A,B)$ of $\Hom_{\Cal C}(A,B)$ for every pair of objects of $\Cal C$, such that: $(1)$ for every objects $A,B,C$ of $\Cal C$, for every $f:A \rightarrow B$ and for every $g:B\rightarrow C$, one has that $gf \in \Cal P(A,C)$ if and only if either $f \in \Cal P(A,B)$ or $g\in \Cal P(B,C)$, and $(2)$ $\Cal P(A,A)$ is a proper subgroup of $\Hom_{\Cal C}(A,A)$ for every object $A$ of $\Cal C$ \cite[p.~565]{AlbPav5}.

Let $A$ be an object of $\Cal C$ and $I$ be a two-sided ideal of the ring $\End_{\Cal C}(A)$. Let $\Cal I$ be the ideal of the category $\Cal C$ defined as follows. A morphism $f \colon X \to Y$ in $\Cal C$ belongs to $\Cal I(X,Y)$ if $\beta f \alpha \in I$ for every pair of morphisms $\alpha \colon
A \to X$ and $\beta \colon Y \to A$ in $\Cal C$. The ideal $\Cal I$ is called the {\em ideal of $\Cal C$ associated to}~$I$ \cite{AlbPav3, AlbPav4}. It is the greatest of the ideals $\Cal I'$ of $\Cal C$ with $\Cal I'(A,A)\subseteq I$. It is easily seen that $\Cal I(A,A) = I$. 
If $A$ is an object of $\Cal C$, the ideals associated to two distinct ideals of $\End_{\Cal C}(A)$ are obviously two distinct ideals of the category~$\Cal C$.

The endomorphism ring $E_B$ of any object (\ref{1}) of $\Cal U$ has $n$ maximal right ideals $I_{B,a,b}$, where $n\le 4$ (Theorem~\ref{EARY}). Let $\Cal I_{B,a,b}$ be the corresponding $n$ ideals of $\Cal E$ associated to the maximal ideals of $E_B$. Set $V(B)\colon=\{\, \Cal I_{B,a,b}\mid I_{B,a,b}$ is a maximal ideal of $E_B\,\}$, so that $V(B)$ has $n\le 4$ elements.

Now set $V(\Cal U)\colon=\bigcup_{B\in\Ob(\Cal U)}V(B)$. Thus $V(\Cal U)$ is a class of ideals of the category~$\Cal E$.

\begin{Lemma} \label{pom} For every $\Cal P\in V(\Cal U)$ and every non-zero object (\ref{2}) of $\Cal E$ such that $E_{B'}$ is semilocal, either $\Cal P(B',B') = E_{B'}$
or $\Cal P(B',B')$ is a maximal ideal of $E_{B'}$. In this second case, the ideal of $\Cal E$ associated to $\Cal P(B',B')$ is equal to $\Cal P$.
\end{Lemma}

\begin{proof} The lemma is an immediate consequence of \cite[Lemma~2.1(ii)]{FacPer1}.\end{proof}

For any ideal $\Cal P$ of $\Cal E$, $F \colon \Cal E \to \Cal E / \Cal P$ will denote the canonical functor. Recall that the factor category $\Cal E/\Cal P$ has the same objects as $\Cal E$ and, for $B,B'\in\Ob(\Cal E)=\Ob(\Cal E/\Cal P)$, the group of morphisms $B\to B'$ in $\Cal E/\Cal P$ is defined to be the factor group $\Hom_{\Cal E}(B,B')/\Cal P(B,B')$.

\begin{Lemma} \label{star} Let (\ref{1}) be an object of $\Cal U$, $I_{B,a,b}$ be a maximal ideal of $E_B$, $\Cal I_{B,a,b}$ its associated ideal in $\Cal E$ and $F \colon \Cal E \to \Cal E/\Cal I_{B,a,b}$ be the canonical functor. Then, for any object (\ref{2}) of $\Cal E$ such that $E_{B'}$ is of finite type, either $F(B') = 0$ or $F(B') \cong  F(B)$. Moreover, if $[B]_{a,b}=[B']_{a,b}$, then $F(B') \cong  F(B)$.
\end{Lemma}

\begin{proof} 
From Lemma \ref{pom}, we know that either $F(B') = 0$ or the endomorphism ring of $F(B')$ is a division ring. Let us consider the latter case. 
As $1_{B'}\notin\Cal I_{B,a,b}(B',B')$, there are $\alpha \colon B \to B'$ and $\beta \colon B' \to B$ such that $\beta \alpha \notin I_{B,a,b}$. Thus $\beta(\alpha\beta)\alpha\notin I_{B,a,b}$. It follows that $\alpha\beta \notin \Cal I_{B,a,b}(B',B')$, 
Therefore $F(\beta)F(\alpha)$  and $F(\alpha)F(\beta)$ are automorphisms, so $F(B) \cong  F(B')$ follows. For the last assertion, note that, also in the case $[B]_{a,b}=[B']_{a,b}$, there are morphisms $\alpha \colon B \to B'$ and $\beta \colon B' \to B$ such that $\beta \alpha \notin I_{B,a,b}$. So, as before, we can conclude that $F(B)\cong F(B')$.
\end{proof}

In the following corollary, we collect some basic properties that can be deduced easily from the previous two lemmas.

\begin{corollary}\label{4.6}
Let (\ref{1}) be an object of $\Cal U$, $I_{B,a,b}$ be a maximal ideal of $E_B$, $\Cal I_{B,a,b}$ its associated ideal in $\Cal E$ and $F \colon \Cal E \to \Cal E/\Cal I_{B,a,b}$ be the canonical functor. 

{\rm (1)} For any object (\ref{2}) of $\Cal U$, the following conditions are equivalent:

\begin{enumerate}
\item[{\rm (a)}]
$F(B')= 0$ in $\Cal E/\Cal I_{B,a,b}$.
\item[{\rm (b)}]
$\Cal I_{B,a,b}(B',B')=E_{B'}$.
\item[{\rm (c)}]
$I_{B',a,b}\subset \Cal I_{B,a,b}(B',B')$.
\item[{\rm (d)}]
$[B]_{a,b}\neq[B']_{a,b}$.
\end{enumerate}

{\rm (2)} For any object (\ref{2}) of $\Cal U$, the following conditions are equivalent:

\begin{enumerate}
\item[{\rm (a)}]
$F(B)\cong F(B')$ in $\Cal E/\Cal I_{B,a,b}$.
\item[{\rm (b)}]
$\Cal I_{B,a,b}(B',B')$ is a proper 
 ideal of $E_{B'}$.
\item[{\rm (c)}]
$I_{B',a,b}= \Cal I_{B,a,b}(B',B')$. 
\item[{\rm (d)}]
$[B]_{a,b}=[B']_{a,b}$.
\end{enumerate}

\end{corollary}

\begin{proof} Proof of (1).
(a)${}\Leftrightarrow{}$(b) $F(B')= 0$ if and only if the endomorphism ring $E_{B'}/\Cal I_{B,a,b}(B',B')$ of $F(B')$ in the factor category $\Cal E/ \Cal I_{B,a,b}$ is the zero ring.

(b)${}\Rightarrow{}$(c) It suffices to note that $I_{B',a,b}$ is always a proper ideal of $E_{B'}$.

(c)${}\Rightarrow{}$(d) Let $\varphi$ be an element in $\Cal I_{B,a,b}(B',B')$ not in $I_{B',a,b}$. For any two morphisms $f\colon B_R\to B'_R$ and $g\colon B'_R\to B_R$ in the category $\Cal E$, one has $g\varphi f \in I_{B,a,b}$, by definition of associated ideal. Since $\varphi \notin I_{B',a,b}$, it follows from \cite[Lemma~1.1]{TAMS} that $gf \in I_{B,a,b}$. This means that $[B]_{a,b}\neq[B']_{a,b}$.

(d)${}\Rightarrow{}$(b)  By Remark \ref{1.1}, if $[B]_{a,b}\ne[B']_{a,b}$, then, for any morphism $f\colon B\to B'$ and $g\colon B' \to B$ in the category $\Cal U$, we have that $g1_{B'} f=g f \in I_{B,a,b}$.
Therefore $1_{B'} \in \Cal I_{B,a,b}(B',B')$, so that $\Cal I_{B,a,b}(B',B')=E_{B'}$.

Proof of (2).
First notice that, by Lemma \ref{star}, either $F(B)=0$ or $F(B)\cong F(B')$ in $\Cal E/\Cal I_{B,a,b}$. Moreover, $I_{B',a,b}$ is always contained in $\Cal I_{B,a,b}(B',B')$. As a matter of fact, suppose $\varphi \in I_{B',a,b}$. By \cite[Lemma~1.1]{TAMS}, for any $f \colon B\to B'$ and any $g \colon B'\to B$, we have that $g\varphi f \in I_{B,a,b}$, so $\varphi \in I_{B,a,b}(B',B')$. Now all the implications follow from  part (1).
\end{proof}

\begin{corollary}\label{cor4.8}
Let {\rm (\ref{1})} and {\rm (\ref{2})} be two objects in the category $\Cal U$. Fix $a=m,e$ and $b=l,u$. Suppose that $[B]_{a,b}=[B']_{a,b}$. Then the following properties hold:
\begin{enumerate}
\item[{\rm (a)}]
$I_{B,a,b}$ is a maximal right ideal of $E_{B}$  if and only if $I_{B',a,b}$ is a maximal right ideal of $E_{B'}$.
\item[{\rm (b)}]
Suppose that $I_{B,a,b}$ is a maximal right ideal of $E_B$. Then, for every $c=m,e$ and $d=l,u$, $I_{B,c,d}= I_{B,a,b}$ if and only if $I_{B',c,d}= I_{B',a,b}$.\end{enumerate}
\end{corollary}

\begin{proof}
{\rm (a)} It suffices to show that $I_{B,a,b}$ maximal implies $I_{B',a,b}$ maximal. From Corollary \ref{4.6} (2), $I_{B',a,b}=\Cal I_{B,a,b}(B',B')$ is a proper ideal of $E_{B'}$, which is a maximal right ideal of $E_{B'}$ by Lemma \ref{pom}.

{\rm (b)} By {\rm (a)}, the hypotheses on (\ref{1}) and (\ref{2}) are symmetrical. Therefore it suffices to show that $I_{B,c,d}= I_{B,a,b}$ implies $I_{B',c,d}= I_{B',a,b}$. The inclusion $I_{B',c,d}\subseteq I_{B',a,b}$ follows from Lemma \ref{4.2}{\rm (a)}. Moreover, by Lemma \ref{4.2} {\rm (b)}, we can interchange the role of $(a,b)$ and $(c,d)$ and deduce the opposite inclusion applying Lemma \ref{4.2} {\rm (a)} again.
\end{proof}

\begin{corollary}\label{inter} Let (\ref{1}) and {\rm (\ref{2})} be two objects of $\Cal U$. Then $V(B)\cap V(B')\neq \emptyset$ if and only if there exists a pair $(p,q)\in\{m,e\}\times\{l,u\}$ such that $I_{B,p,q}$ is a maximal right ideal of $E_B$, $I_{B',p,q}$ is a maximal ideal of $E_{B'}$ and $\Cal I_{B,p,q}=\Cal I_{B',p,q}$. Moreover, for such a pair, $[B]_{p,q}=[B']_{p,q}$.
\end{corollary}

\begin{proof}
Let $\Cal P \in V(B)\cap V(B')$. Then there exist two pairs $(p,q)$ and $(r,s)$ in $\{m,e\}\times\{l,u\}$ such that $I_{B,p,q}$ is a maximal right ideal of $E_B$, $I_{B',r,s}$ is a maximal ideal of $E_{B'}$ and $\Cal P=\Cal I_{B,p,q}=\Cal I_{B',r,s}$ is the ideal of $\Cal E$ associated both to $I_{B,p,q}$ and $I_{B',r,s}$. We have $I_{B',p,q}\subseteq \Cal I_{B,p,q}(B',B')=\Cal I_{B',r,s}(B',B')=I_{B',r,s}$. If the inclusion is proper, then $I_{B',r,s}=E_{B'}$  by Corollary \ref{4.6} (1), a contradiction. It follows that $I_{B',p,q}=I_{B',r,s}$, so we can replace the pair $(r,s)$ with $(p,q)$.
Clearly, the converse holds as well. Finally, the last assertion follows from Corollary \ref{4.6}~(2). 
\end{proof}

According to \cite{FacPer1}, we say that an ideal $\Cal M$ of a preadditive category $\Cal C$ is {\it maximal} if the improper ideal of $\Cal C$ is the unique ideal of the category $\Cal C$ properly containing $\Cal M$. Moreover, we say that a preadditive category $\Cal C$ is {\it simple} \cite{FacPer1} if it has exactly two ideals, necessarily the trivial ones. Hence, a simple category has non-zero objects.

In the category $\Cal U$, $V(\Cal U)$ turns out to be a class of maximal ideals of $\Cal U$, by \cite[Lemma~2.4]{FacPer1}.

\begin{Prop}\label{4.7}
Let (\ref{1}) be an object of $\Cal U$, $I_{B,a,b}$ be a maximal ideal of $E_B$ and $\Cal I_{B,a,b}$ its associated ideal in $\Cal U$. Then the factor category $\Cal U/\Cal I_{B,a,b}$ is simple and the endomorphism ring of all its non-zero objects are division rings.
\end{Prop}

\begin{proof}
The Proposition is an immediate consequence of \cite[Theorem~3.2]{FacPer1}.
\end{proof}

\begin{corollary}\label{4.7,5} Let (\ref{1}) be an object of $\Cal U$, $I_{B,a,b}$ be a maximal ideal of $E_B$ and $\Cal I_{B,a,b}$ its associated ideal in $\Cal E$. If $F(B)^t\cong F(B)^s$ in the factor category $\Cal E/\Cal I_{B,a,b}$ for integers $t,s\ge 0$, then $t=s$.\end{corollary}

\begin{proof}
Let $D$ be the endomorphism ring of the object $F(B)$ of $\Cal E/\Cal I_{B,a,b}$. Apply the functor $\Hom(F(B),-)\colon \Cal E/\Cal I_{B,a,b}\to\Mod D$.
\end{proof}

Notice that we are using the following notations:

$I_{B,a,b}$ is an ideal of $E_B$, the endomorphism ring of an object of $\Cal E$;

$\Cal I_{B,a,b}$ is the ideal of $\Cal E$, or possibly a full subcategory of $\Cal E$, associated to $I_{B,a,b}$;

$\Cal K_{a,b}$ is one of four completely prime ideals of $\Cal U$, defined, for every pair of objects (\ref{1}), (\ref{2}) of $\Cal U$, by
\newline
$\Cal K_{m,l}((\ref{1}), (\ref{2})):=\{\,f\colon B\to B' \mid \ker(f)\cap A\ne 0\,\}$,
\newline
$\Cal K_{e,l}((\ref{1}), (\ref{2})):=\{\,f\colon B\to B'  \mid f(A)\subset A'\,\}$,
\newline
$\Cal K_{m,u}((\ref{1}), (\ref{2})):=\{\,f\colon B\to B'  \mid A\subset f^{-1}(A')\,\}$ and 
\newline
$\Cal K_{e,u}((\ref{1}), (\ref{2})):=\{\,f\colon B\to B'  \mid f(B)+A'\subset B'\,\}$.

\begin{Lemma}\label{nuovo}
Let
$$
\xymatrix@1{ 0 \ar[r] &A_i \ar[r] &B_i \ar[r] &C_i \ar[r] &0,}\qquad i=1,2,\dots,n,
$$
be objects of $\Cal U$. Then every maximal two-sided ideal of $E_{\oplus_{j=1}^n B_j}$ is of the form
$\Cal I_{B_k,a,b}(\oplus_{j=1}^n B_j,\oplus_{j=1}^n B_j)$
for some $k=1,\dots,n$, $a=m,e$ and $b=l,u$ such that $I_{B_k,a,b}$ is a maximal right ideal of $E_{B_k}$. Conversely, if $(k,a,b)$ is a triple such that $k=1,\dots,n$, $a=m,e$, $b=l,u$ and $I_{B_k,a,b}$ is a maximal right ideal of $E_{B_k}$, then $\Cal I_{B_k,a,b}(\oplus_{j=1}^n B_j,\oplus_{j=1}^n B_j)$ is a maximal two-sided ideal of $E_{\oplus_{j=1}^n B_j}$.
\end{Lemma}

\begin{proof}
Let $I$ be a maximal two-sided ideal of $E_{\oplus_{j=1}^n B_j}$ and let $\Cal I$ be the ideal of $\Cal E$ associated to $I$. Using \cite[Lemma~2.1(ii)]{FacPer1}, we get that for any $k=1,\dots,n$, either $\Cal I(B_k,B_k)=E_{B_k}$ or $\Cal I(B_k,B_k)$ is a maximal two-sided ideal of $E_{B_k}$. If $\Cal I(B_k,B_k)=E_{B_k}$ for all $k=1,\dots,t$, by the definition of associated ideal it follows that $\varepsilon_k \pi_k \in I$ for every $k=1,\dots,n$, where $\varepsilon_k \colon B_k\rightarrow \oplus_{j=1}^n B_j$ and $\pi_k\colon \oplus_{j=1}^n B_j \rightarrow B_k$ are the embedding and the canonical projection respectively. In particular $1_{\oplus_{j=1}^n B_j}=\sum_{k=1}^n \varepsilon_k \pi_k \in I$, which is a contradiction. It follows that there exists a triple $(k,a,b)$ such that $\Cal I(B_k,B_k)=I_{B_k,a,b}$ is a maximal right ideal of $E_{B_k}$. By \cite[Lemma~2.1(i)]{FacPer1}, $\Cal I_{B_k,a,b}=\Cal I$, so that, in particular, $I=\Cal I(\oplus_{j=1}^n B_j,\oplus_{j=1}^n B_j)=\Cal I_{B_k,a,b}(\oplus_{j=1}^n B_j,\oplus_{j=1}^n B_j)$.

Conversely,  let $(k,a,b)$ be a triple such that $k=1,\dots,n$, $a=m,e$, $b=l,u$ and $I_{B_k,a,b}$ is a maximal right ideal of $E_{B_k}$. By \cite[Lemma~2.1(ii)]{FacPer1}, either $\Cal I_{B_k,a,b}(\oplus_{j=1}^n B_j,\oplus_{j=1}^n B_j)$ is a maximal ideal of $E_{\oplus_{j=1}^n B_j}$ or $$\Cal I_{B_k,a,b}(\oplus_{j=1}^n B_j,\oplus_{j=1}^n B_j)=E_{\oplus_{j=1}^n B_j}.$$ In this second case, the identity of $\oplus_{j=1}^n B_j$ is in $\Cal I_{B_k,a,b}(\oplus_{j=1}^n B_j,\oplus_{j=1}^n B_j)$, so that the identity of $B_k$ is in $\Cal I_{B_k,a,b}(B_k,B_k)=I_{B_k,a,b}$, a contradiction.
\end{proof}

\begin{remark}\label{nuovo'}{\rm If $T$ is any semilocal ring and $M$ is a maximal two-sided ideal of~$T$, then there exists an element $\delta_M\in T$ such that $\delta_M\equiv 1_T\pmod{M}$ and $\delta_M\equiv 0\pmod{N}$ for every other maximal two-sided ideal $N$ of $T$ different from $M$. This follows from the fact that $T/J(T)$ is a direct product of finitely many simple rings.}\end{remark}

\begin{Prop}\label{parziale}
Let \begin{equation}\label{1'}
\xymatrix@1{ 0 \ar[r] &A_i \ar[r] &B_i \ar[r] &C_i \ar[r] &0,}\qquad i=1,2,\dots,n,
\end{equation}
and
\begin{equation}\label{2'}
\xymatrix@1{ 0 \ar[r] &A'_j \ar[r] &B'_j \ar[r] &C'_j \ar[r] &0,}\qquad j=1,2,\dots,m,
\end{equation}
be $n+m$ objects in the category $\Cal E$ with all the modules $A_i,C_i,A'_j,C'_j$ non-zero uniserial modules. For every $a=m,e$ and every $b=l,u$, set $$
X_{a,b}:=\{\,i\mid i=1,2,\dots,n,\ I_{B_i,a,b}\ \mbox{is a maximal ideal of }E_{B_i}\,\}
$$
and
$$
Y_{a,b}:=\{\,j\mid j=1,2,\dots,m,\ I_{B'_j,a,b}\ \mbox{is a maximal ideal of }E_{B'_i}\,\}.
$$
Then $\bigoplus_{i=1}^nB_i\cong\bigoplus_{j=1}^mB'_j$ in the category $\Cal E$ if and only if there exist four bijections $\varphi_{a,b}\colon X_{a,b}\to Y_{a,b}$, $a=m,e$ and $b=l,u$, with $[B_i]_{a,b}=[B'_{\varphi_{a,b}(i)}]_{a,b}$ for every $i\in X_{a,b}$.
\end{Prop}

\begin{proof}
$(\Rightarrow)$ Suppose there exists an isomorphism $\bigoplus_{i=1}^nB_i\cong\bigoplus_{j=1}^mB'_j$ in $\Cal E$. Fix $a=m,e$ and $b=l,u$. Apply Proposition~\ref{4.7} to any fixed object $B''$ of $\Cal U$ for which $I_{B'',a,b}\ \mbox{is a maximal ideal of }E_{B''_i}$. By Proposition~\ref{4.7}, the corresponding factor category $\Cal U/\Cal I_{B'',a,b}$ is simple. Also, either $[B_i]_{a,b}=[B'']_{a,b}$ or $[B_i]_{a,b}\ne[B'']_{a,b}$. By Corollary~\ref{4.6}, we have that either $F(B_i)\cong F(B'')$ or $F(B_i)=0$ in the category $\Cal U/\Cal I_{B'',a,b}$. Thus, in the category $\Cal E/\Cal I_{B'',a,b}$, we have that $F(\bigoplus_{i=1}^nB_i)\cong F(B'')^{t_{B'',a,b}}$, where $t_{B'',a,b}=|\{\,i\mid i=1,2,\dots,n,\ [B_i]_{a,b}=[B'']_{a,b}\,\}|$. Similarly, $F(\bigoplus_{j=1}^mB'_j)\cong F(B'')^{s_{B'',a,b}}$, where $s_{B'',a,b}=|\{\,j\mid i=1,2,\dots,m,\ [B_j]_{a,b}=[B'']_{a,b}\,\}|$. Then,  by Corollary~\ref{4.7,5}, $t_{B'',a,b}=s_{B'',a,b}$ for every object $B''$ of $\Cal U$  for which $I_{B'',a,b}\ \mbox{is a maximal ideal of }E_{B''_i}$. Thus there is a bijection $\varphi_{B'',a,b}$ between $\{\,i\mid i=1,2,\dots,n,\ [B_i]_{a,b}=[B'']_{a,b}\,\}$ and $\{\,j\mid i=1,2,\dots,m,\ [B_j]_{a,b}=[B'']_{a,b}\,\}$ for every $B''$ in $\Cal U$ and every $a,b$ with $I_{B'',a,b}$ a maximal ideal of $E_{B''}$. Now, for $[B_i]_{a,b}=[B'']_{a,b}$, $I_{B'',a,b}$ is a maximal ideal of $E_{B''}$ if and only if $I_{B_i,a,b}$ is a maximal ideal of $E_{B_i}$ (Corollary~\ref{cor4.8}~{\rm (a)}). Hence $\varphi_{B'',a,b}$ is a bijection between $\{\,i\mid i=1,2,\dots,n,\ [B_i]_{a,b}=[B'']_{a,b}$ and $I_{B_i,a,b}\ \mbox{is a maximal ideal of }E_{B_i}\,\}$ and $\{\,j\mid i=1,2,\dots,m,\ [B'_j]_{a,b}=[B'']_{a,b}$ and $I_{B'_j,a,b}\ \mbox{is a maximal ideal of }E_{B'_j}\,\}$ for every $B''$ in $\Cal U$.
Glueing together all these bijections $\varphi_{B'',a,b}$, we get the required bijection $\varphi_{a,b}\colon X_{a,b}\to Y_{a,b}$.

$(\Leftarrow)$ Conversely, suppose that four bijections $\varphi_{a,b}$  with the property in the statement of the Proposition exist. Fix any object $B''$ in $\Cal U$ with $I_{B'',a,b}$ a maximal ideal of $E_{B''}$.  Then the number of indices $i$ with $i=1,2,\dots,n,\ I_{B_i,a,b}$ a maximal ideal of $E_{B_i}$ and $[B_i]_{a,b}=[B'']_{a,b}$
is equal to the number of indices $j$ with $j=1,2,\dots,m,\ I_{B'_j,a,b}$ a maximal ideal of $E_{B'_i}$ and $[B'_j]_{a,b}=[B'']_{a,b}$. By Corollary~\ref{cor4.8}, the number of indices $i$ with $i=1,2,\dots,n$ and $[B_i]_{a,b}=[B'']_{a,b}$
is equal to the number of indices $j$ with $j=1,2,\dots,m$ and $[B'_j]_{a,b}=[B'']_{a,b}$. 
Call $n_{B'',a,b}$ this number. Let $\Cal I_{B'',a,b}$ be the ideal in the category $\Cal E$ associated to $I_{B'',a,b}$ and $F\colon\Cal E\to\Cal E/\Cal I_{B'',a,b}$ be the canonical functor. Then, for every short exact sequence $B$ in $\Cal U$, $F(B)=0$ if $[B]_{a,b}\ne[B'']_{a,b}$, and $F(B)\cong F(B'')$ if $[B]_{a,b}=[B'']_{a,b}$ (Corollary~\ref{4.6}). Thus $F(\bigoplus_{i=1}^nB_i)\cong F(B'')^{n_{B'',a,b}}$ and $F(\bigoplus_{j=1}^mB'_j)\cong F(B'')^{n_{B'',a,b}}$, so that $F(\bigoplus_{i=1}^nB_i)\cong F(\bigoplus_{j=1}^mB'_j)$ in $\Cal E/\Cal I_{B'',a,b}$. In particular, this occurs for all the short exact sequences $B''\in\{B_1,\dots,B_n,B'_1,\dots,B'_m\}$ and the pairs $(a,b)$ with $I_{B'',a,b}$ is a maximal ideal of $E_{B''}$. 

From the isomorphism $F(\bigoplus_{i=1}^nB_i)\cong F(B'')^{n_{B'',a,b}}$, we get that the ring\linebreak $\End_{\Cal E/\Cal I_{B'',a,b}}(F(\bigoplus_{i=1}^nB_i))$ is isomorphic to the ring $M_{n_{B'',a,b}}(E_{B''}/I_{B'',a,b})$ of all $n_{B'',a,b}\times n_{B'',a,b}$ matrices with entries in the division ring $E_{B''}/I_{B'',a,b}$. This ring of matrices is a simple artinian ring. Thus the canonical functor $F$ induces a surjective ring morphism $\End_{\Cal E}(\bigoplus_{i=1}^nB_i)\to \End_{\Cal E/\Cal I_{B'',a,b}}(F(\bigoplus_{i=1}^nB_i))$ onto a simple ring, and its kernel is therefore a maximal two-sided ideal of $\End_{\Cal E}(\bigoplus_{i=1}^nB_i)$. This kernel is $\Cal I_{B'',a,b}(\bigoplus_{i=1}^nB_i,\bigoplus_{i=1}^nB_i)$, which is therefore a maximal two-sided ideal of the semilocal ring $\End_{\Cal E}(\bigoplus_{i=1}^nB_i)$ (cf.~Lemma~\ref{nuovo}). By Remark~\ref{nuovo'}, there exists an endomorphism $\delta_{[B'']_{a,b}}$ of $\bigoplus_{i=1}^nB_i$ such that $\delta_{[B'']_{a,b}}\equiv 1$ modulo this maximal two-sided ideal  and $\delta_{[B'']_{a,b}}$ belongs to all the other maximal two-sided ideals of 
$\End_{\Cal E}(\bigoplus_{i=1}^nB_i)$. Similarly, we find endomorphisms $\delta'_{[B'']_{a,b}}$ of $\bigoplus_{j=1}^mB_j'$.

Since $F(\bigoplus_{i=1}^nB_i)\cong F(\bigoplus_{j=1}^mB'_j)$, there exists a morphism $f_{B'',a,b}\colon \bigoplus_{i=1}^nB_i\to\bigoplus_{j=1}^mB'_j$ in $\Cal E$, which becomes an isomorphism in $\Cal E/\Cal I_{B'',a,b}$. Consider the morphism $f:=\sum\delta'_{[B'']_{a,b}}f_{B'',a,b}\delta_{[B'']_{a,b}}\colon \bigoplus_{i=1}^nB_i\to\bigoplus_{j=1}^mB'_j$, where the sum ranges over the maximal two-sided ideals in the endomorphism ring (Lemma~\ref{nuovo}), so that the number of summands in the definition of $f$ is equal to the number of maximal two-sided ideals of $E_{\bigoplus_{i=1}^nB_i}\cong E_{\bigoplus_{j=1}^mB'_j}$. Thus $f$ becomes an isomorphism in the factor category $\Cal E/\Cal I_{B'',a,b}$ for any triple $(B'',a,b)$ such that $B''\in\{B_1,\dots,B_n,B'_1,\dots,B'_m\}$, $a\in\{m,e\}$, $b\in\{l,u\}$ and $I_{B'',a,b}$ is a maximal ideal of $E_{B''}$. By
\cite[Proposition 5.2]{Adel}, $f$ is an  isomorphism in the category $\Cal E/\Cal I$, where $\Cal I$ is the intersection of these finitely many ideals $\Cal I_{B'',a,b}$. The restriction of $\Cal I$ to the full additive subcategory $\Cal S:=\add(B_1\oplus\dots\oplus B'_m)$ of $\Cal E$ is the Jacobson radical $\Cal J(\Cal S)$ of $\Cal S$. Hence $\bigoplus_{i=1}^nB_i$ and $\bigoplus_{j=1}^mB'_j$ are isomorphic in the category $\Cal E/\Cal I$. Thus there are morphisms $\alpha\colon \bigoplus_{i=1}^nB_i\to\bigoplus_{j=1}^mB'_j$ and $\beta\colon\bigoplus_{j=1}^mB'_j\to\bigoplus_{i=1}^nB_i$ in $\Cal E$ such that $\alpha\beta\equiv 1$ and $\beta\alpha\equiv1$ modulo $\Cal I$. Equivalently, the endomorphisms $\alpha\beta$ and $\beta\alpha$ in $\Cal S$ are congruent to $1$ modulo $J(\End_{\Cal S}(\bigoplus_{j=1}^mB'_j))$ and $J(\End_{\Cal S}(\bigoplus_{i=1}^nB_i))$, respectively. Thus $\alpha\beta$ and $\beta\alpha$ are invertible in the rings $\End_{\Cal S}(\bigoplus_{j=1}^mB'_j)$ and $\End_{\Cal S}(\bigoplus_{i=1}^nB_i)$, respectively. In particular, $\alpha$ is both right invertible and left invertible in the category $\Cal E$. It follows that $\alpha$ is an isomorphism in $\Cal E$.
\end{proof}

\begin{remark} In the statement of the previous proposition, we have used the condition ``$I_{B,a,b}\ \mbox{is a maximal ideal of }E_{B}$''. This condition can be expressed very easily in terms relative to $B$. For instance, suppose we want to express the condition 
$I_{B,m,l}\ \mbox{is a maximal ideal of }E_{B}$. By Theorem~\ref{EARY}, the ideals $I_{B,m,l},I_{B,e,l},I_{B,m,u}$ and $I_{B,e,u}$ are completely prime ideals, and the maximal ideals of $E_B$ are some of them. Thus, by Remark~\ref{rem}, 
$I_{B,m,l}\ \mbox{is a maximal ideal of }E_{B}$ if and only if $I_{B,m,l}\not\subseteq I_{B,m,u} \cup I_{B,e,l} \cup I_{B,e,u}$, that is, if and only if there exists an endomorphism $f$ of (\ref{1}) with $\ker(f)\cap A\ne 0$, $f(A)= A$,
$A= f^{-1}(A)$ and $f(B)+A= B$.
\end{remark}

In the final result of this section, we will make use of some techniques, notations and ideas taken from \cite{DF}. If $X$ and $Y$ are finite disjoint sets, we will denote by $D(X,Y;E)$ the bipartite
digraph (= directed graph) having $X$ and $Y$ as disjoint sets of non-adjacent vertices and $E$ as set of
edges. Equivalently, $V =X\cup Y$ is the vertex set of $D(X,Y;E)$, $ E\subseteq X \times Y \cup Y \times X$ is the set
of its edges, and $X \cap Y =\emptyset$. For every subset $T\subseteq V$, let $N^+(T) =\{\,w\in V\mid (v,w)\in E$ 
for some $v\in T\,\}$ be the {\em out-neighborhood} of $T$ (\cite[introduction, p.184]{DF}). Define an equivalence relation $\sim_s$ on
$V$ by $v\sim_s w$ if there are both an oriented path from $v$ to $w$ and an oriented path from $w$ to $v$ ($v,w\in V$).

\begin{Prop}\label{AL} (\cite[Lemma~2.1]{DF}, Krull-Schmidt Theorem for bipartite digraphs). Let $X$ and $Y$ be disjoint
sets of cardinality $n$ and $m$, respectively. Set $V :=X\cup Y$. Let $D = D(X,Y ; E)$
be a bipartite digraph having $X$ and $Y$ as disjoint sets of non-adjacent vertices. If
$|T|\le|N^+(T)|$ for every subset $T$ of $V$, then $n = m$ and, after a suitable relabeling of the indices of the elements $x_1,\dots,x_n$ of $X$ and  $y_1,\dots,y_n$ of $Y$, $x_i \sim_s y_i$ for every $i = 1,\dots,n$.\end{Prop}

\begin{Prop}\label{completo} Let \begin{equation}
\label{1'x}
\xymatrix@1{ 0 \ar[r] &A_i \ar[r] &B_i \ar[r] &C_i \ar[r] &0,}\qquad i=1,2,\dots,n,
\end{equation}
and
\begin{equation}\label{2x'}
\xymatrix@1{ 0 \ar[r] &A'_j \ar[r] &B'_j \ar[r] &C'_j \ar[r] &0,}\qquad j=1,2,\dots,m,
\end{equation}
be $n+m$ objects in the category $\Cal E$ with all the modules $A_i,C_i,A'_j,C'_j$ non-zero uniserial modules. Then $\bigoplus_{i=1}^nB_i\cong\bigoplus_{j=1}^mB'_j$ in the category $\Cal E$ if and only if $n=m$ and there exist four permutations $\varphi_{a,b}$ of $\{1,2,\dots,n\}$, $(a,b)\in\{m,e\}\times\{l,u\}$, with $[B_i]_{a,b}=[B'_{\varphi_{a,b}(i)}]_{a,b}$ for every $i=1,2,\dots,n$, every $a=m,e$ and every $b=l,u$.\end{Prop}

\begin{proof} 
Suppose $\bigoplus_{i=1}^nB_i\cong\bigoplus_{j=1}^mB'_j$ in the category $\Cal E$. Then $\bigoplus_{i=1}^nA_i\cong\bigoplus_{j=1}^mA'_j$ in the category $\Mod R$, so that, taking the Goldie dimension, we get that $n=m$. Let $\alpha\colon  \bigoplus_{i=1}^nB_i \rightarrow \bigoplus_{j=1}^mB'_j$ be an isomorphism in $\Cal E$ with inverse $\beta\colon  \bigoplus_{j=1}^m B'_j \rightarrow \bigoplus_{i=1}^n B_i$. Denote by $\varepsilon_h\colon B_h\rightarrow \bigoplus_{i=1}^nB_i$, $\pi_h\colon  \bigoplus_{i=1}^nB_i \rightarrow B_h$, $\varepsilon'_k\colon B'_k\rightarrow \bigoplus_{j=1}^nB'_j$ and $\pi'_k\colon  \bigoplus_{j=1}^nB'_j \rightarrow B'_k$ the embeddings and the canonical projections and consider the composite morphisms $\chi_{h,k}:=\pi'_k\alpha \varepsilon_h\colon B_h\rightarrow B'_k$ and $\chi'_{k,h}: =\pi_h\beta\varepsilon'_k\colon B'_k\rightarrow B_h$. We will prove the existence of the permutation $\varphi_{e,u}$. The proof of the existence of the other three permutations $\varphi_{a,b}$ is similar. Define a bipartite digraph $D=D(X,Y;E)$ having $X=\{B_1,\dots,B_n\}$ and $Y=\{B'_1,\dots,B'_n\}$ as disjoint sets of non-adjacent vertices, and the set $E$ of edges defined as follows: one edge from $B_h$ to $B'_k$ for each $h$ and $k$ such that the induced mapping $\overline{\chi_{h,k}}\colon  C_h\rightarrow C'_k$ is surjective, and one edge from $B'_k$ to $B_h$ for each $h$ and $k$ such that the induced mapping $\overline{\chi'_{k,h}}\colon C'_k\rightarrow C_h$ is surjective. We want to show that $|T|\leq |N^+(T)|$ for every subset $T\subseteq X\cup Y$ of vertices. Since the digraph is bipartite, we can suppose that $T\subseteq X$. If $r=|T|$ and $s=|N^+(T)|$, relabeling the indices we may assume that $T=\{B_1,\dots, B_r\}$ and $N^+(T)=\{B'_1,\dots,B'_s\}$. This means that the induced morphisms $\overline{\chi_{h,k}}$ are not surjective for every $h=1,\dots, r$  and every $k=s+1,\dots, n$. Since the modules $C'_k$ are all uniserial, we have that $L_k:=\sum_{h=1}^r \overline{\chi_{h,k}}(C_h)\subset C'_k$ for every $k=s+1,\dots,n$, and therefore  the quotient modules $C'_k/L_k$ are non-zero for every $k=s+1,\dots,n$. Let $\pi\colon  \bigoplus_{j=1}^n C'_j \rightarrow \bigoplus_{k=s+1}^n C'_k/L_k$ be the canonical projection. For every $h=1,\dots, r$ and for every $k=s+1,\dots,n$, the composite morphism
$$
C_h\overset{\overline{\varepsilon_h}}{\longrightarrow} \bigoplus_{i=1}^n C_i
\overset{\overline{\alpha}}{\longrightarrow} \bigoplus_{j=1}^n C'_j
\overset{\overline{\pi'_k}}{\longrightarrow} C'_k \rightarrow C'_k/L_k
$$
is zero because $\overline{\pi'_k}\overline{\alpha}\overline{\varepsilon_h}(C_h)=\overline{\chi_{h,k}}(C_h)\subseteq L_k$. It follows that, for every $h=1,\dots,r$, $\overline{\varepsilon_h}(C_h)$ is contained in the kernel of the composite mapping
$$
\bigoplus_{i=1}^n C_i
\overset{\overline{\alpha}}{\longrightarrow} \bigoplus_{j=1}^n C'_j
\overset{\overline{\pi'_k}}{\longrightarrow} C'_k \rightarrow C'_k/L_k
$$
for every $k=s+1,\dots,n$. Since $\sum_{h=1}^r \overline{\varepsilon_h}(C_h)=\bigoplus_{h=1}^r C_h$, it follows that there exists a morphism $\bigoplus_{i=1}^n C_i/\bigoplus_{h=1}^r C_h\cong \bigoplus_{\ell=r+1}^n C_\ell \rightarrow C'_k/L_k$ making the following diagram
$$
\xymatrix{
\bigoplus_{i=1}^n C_i \ar[r] \ar[d]_{\overline{\alpha}}
 &  \bigoplus_{\ell=r+1}^n C_\ell \ar[d] \\
\bigoplus_{j=1}^n C'_j \ar[r] & C'_k/L_k
}
$$
commute for every $k=s+1,\dots, n$. Therefore there exist a morphism $$\gamma\colon  \bigoplus_{\ell=r+1}^n C_\ell \rightarrow \bigoplus_{k=s+1}^n C'_k/L_k$$ and a commutative diagram 
$$
\xymatrix{
\bigoplus_{i=1}^n C_i \ar[r] \ar[d]_{\overline{\alpha}} &  \bigoplus_{\ell=r+1}^n C_\ell \ar[d]^{\gamma} \\
\bigoplus_{j=1}^n C'_j \ar[r]^-{\pi} & \bigoplus_{k=s+1}^n C'_k/L_k.
}
$$ Since the horizontal arrows are the canonical projections and $\overline{\alpha}$ is an isomorphism, the morphism $\gamma$ must be surjective. Taking the dual Goldie dimension of the domain and the codomain of $\gamma$, we get that $n-r\geq n-s$, that is, $|T|=r\leq s=|N^+(T)|$, as we wanted to prove. Now apply Proposition~\ref{AL} to the digraph $D$. This concludes the proof of one of the implications in the statement of our proposition.

Conversely, assume that $n=m$ and there exist four permutations $\varphi_{a,b}$ of $\{1,2,\dots,n\}$, $(a,b)\in\{m,e\}\times\{l,u\}$, with $[B_i]_{a,b}=[B'_{\varphi_{a,b}(i)}]_{a,b}$ for every $i=1,2,\dots,n$, every $a=m,e$ and every $b=l,u$.
By Corollary~\ref{cor4.8} {\rm (a)}, $\varphi_{a,b}(X_{a,b})=Y_{a,b}$. Thus $\varphi_{a,b}$ restricts to a bijection $\varphi'_{a,b}$ of $X_{a,b}$ onto $Y_{a,b}$. We can now apply Proposition~\ref{parziale}, getting that $\bigoplus_{i=1}^nB_i$ and $\bigoplus_{j=1}^mB'_j$ are isomorphic objects in the category $\Cal E$.
\end{proof}

\section{Examples}

\begin{example}
Let $R$ be a ring having two non-isomorphic simple right $R$-modules $S$ and $S'$. Consider the following two objects of $\Cal U$:
$$
\xymatrix@1{ 0 \ar[r] & S \ar[r]^-{\varepsilon_1} & S\oplus S' \ar[r]^-{\pi_1} &S' \ar[r] &0}
$$
and
$$
\xymatrix@1{ 0 \ar[r] &S' \ar[r]^-{\varepsilon_2} &S\oplus S' \ar[r]^-{\pi_2} &S \ar[r] &0,}
$$
where $\varepsilon_i$ and $\pi_j$ are the embeddings and the canonical projections, respectively. These two objects have the same endomorphism ring in $\Cal E$:
$$
E_B=
\left( \begin{array}{cc}
\End_R(S) & 0 \\
0 & \End_R(S')
\end{array} \right)
=E_{B'}.
$$
Moreover, the maximal right ideals of $E_B=E_{B'}$ are:
$$
I_{B,m,u}=I_{B,e,u}=
\left( \begin{array}{cc}
\End_R(S) & 0 \\
0 & 0
\end{array} \right)
=I_{B',m,l}=I_{B',e,l}
$$
and
$$
I_{B,m,l}=I_{B,e,l}=
\left( \begin{array}{cc}
0 & 0 \\
0 & \End_R(S')
\end{array} \right)
=I_{B',m,u}=I_{B',e,u}.
$$
So, for any fixed $(a,b)\in\{m,e\}\times\{l,u\}$, $I_{B,a,b}$ is a maximal right ideal of $E_B$ and $I_{B',a,b}$ is a maximal ideal of $E_{B'}$, but $[B]_{a,b}\neq[B']_{a,b}$, because otherwise $S$ and $S'$ would be isomorphic, contrary to the hypothesis. In particular, the two objects are not isomorphic in $\Cal E$ (Proposition \ref{propiso}). Notice that, in the notation of  Proposition \ref{parziale}, the unique bijections $\varphi_{a,b}\colon X_{a,b}\rightarrow Y_{a,b}$ do not satisfy the ``compatibility'' conditions $[B_i]_{a,b}=[B'_{\varphi_{a,b}(i)}]_{a,b}$.
\end{example}

\begin{example}
Let $R$ be a ring and let $U,V$ be two non-zero uniserial right $R$-modules. Consider the following object of $\Cal U$:
$$
\xymatrix@1{ 0 \ar[r] & U \ar[r]^-{\varepsilon} & U\oplus V \ar[r]^-{\pi} &V \ar[r] &0,}
$$
where $\varepsilon\colon u \mapsto (u,0)$ is the embedding and $\pi\colon(u,v)\mapsto v$ is the canonical projection. Its endomorphism ring in $\Cal E$ is
$$
E_B=
\left( \begin{array}{cc}
\End_R(U) & 0 \\
\Hom_R(V,U) & \End_R(V)
\end{array} \right).
$$
For any endomorphism $f=\left(\begin{smallmatrix} f_{1,1} & 0\\ f_{2,1}& f_{2,2} \end{smallmatrix}\right) \in E_B$, we have $f\mid^U_U=f_{1,1}$ and $\overline{f}=f_{2,2}$. So, using the notation at the beginning of Section \ref{UM}, $f \in I_{B,m,l}$ (resp. $f \in I_{B,e,l}$) if and only if $f_{1,1} \in I_U$ (resp. $f_{1,1} \in K_U$). Similarly, $f \in I_{B,m,u}$ (resp. $f \in I_{B,e,u}$) if and only if $f_{2,2} \in I_V$ (resp. $f_{2,2} \in K_V$). In particular, the type of $E_B$ is exactly the sum of the types of $\End_R(U)$ and $\End_R(V)$. Therefore, choosing suitable uniserials $R$-modules, we can produce objects of $\Cal U$ whose endomorphism ring has exactly 2, 3 or 4 maximal right ideals. 
\end{example}

\begin{example}  In \cite[Example 2.1]{TAMS}, the second author constructed examples of uniserial right $R$-modules $U_{i,j}$ to show that a module that is a direct sum of $n$ uniserial modules can have $n!$ pair-wise non-isomorphic direct-sum decompositions into indecomposables. All those modules $U_{i,j}$ have exactly two maximal ideals and we will now show that they are extensions of two uniserial modules with local endomorphism ring. The construction of the ring $R$ and the uniserial right $R$-modules $U_{i,j}$ was the following.

Let ${\bf M}_n({\Q})$ be the ring of all $n\times n$-matrices over 
the field $\Q$ of rational numbers. Let $\Z$ be the ring of integers and let
${\Z}_p,{\Z}_q$ be
the localizations of $\Z$
at two distinct maximal ideals $(p)$ and $(q)$ of $\Z$ (here $p,q\in{\Z}$ are distinct prime
numbers).
Let $\Lambda_p$ denote the subring of ${\bf M}_n({\Q})$ whose elements are the $n\times n$-matrices with
entries in ${{\Z}_p}$ on and above the diagonal
and entries in ${p{\Z}_p}$ 
under the diagonal, that is,
$$\Lambda_p:=\cmat{cccc} {{\Z}_p}&{{\Z}_p}&\dots&{{\Z}_p}\\
p{{\Z}_p}&{{\Z}_p}&\dots&{{\Z}_p}\\
\vdots& &\ddots&\\
p{{\Z}_p}&p{{\Z}_p}&\dots&{{\Z}_p}
\fmat\subseteq {\bf M}_n({\Q}).
$$
Similarly, set
$$\Lambda_q:=\cmat{cccc} {{\Z}_q}&{{\Z}_q}&\dots&{{\Z}_q}\\ q{{\Z}_q}&{{\Z}_q}&\dots&{{\Z}_q}\\
\vdots& &\ddots&\\
q{{\Z}_q}&q{{\Z}_q}&\dots&{{\Z}_q}
\fmat\subseteq {\bf M}_n({\Q}).
$$  
If $$R:=\cmat{cc}\Lambda_p&0\\
{\bf M}_n({\Q})&\Lambda_q\fmat,$$ then $R$ is a subring of the ring
${\bf M}_{2n}({\Q})$ of $2n\times 2n$-matrices with rational entries.

For every $i=1,2,\dots,n$, set
$$
V_i:=(
\underbrace{{\Q},\dots,{\Q}}_{n},
\underbrace{q{\Z}_{{q}},\dots,q{\Z}_{{q}}}_{i-1},
\underbrace{{\Z}_{{q}},\dots,{\Z}_{{q}}}_{n-i+1})
$$
and 
$$
X_i:=(
\underbrace{p{\Z}_{{p}},\dots,p{\Z}_{p}}_{i-1},
\underbrace{{\Z}_{{p}},\dots,{\Z}_{{p}}}_{n-i+1},
\underbrace{0,\dots,0}_{n}).
$$

Then $V_i$ and the $X_j$ are right $R$-modules with right $R$-module structure given by matrix multiplication. Moreover, the unique infinite composition series of $V_1$ is
\begin{align*}
&V_1\supset V_2\supset \dots\supset V_n\supset qV_1\supset \dots \supset q V_n \\&
\supset q^2V_1 \supset \dots \supset q^2 V_n \supset 
\dots \supset p^{-2}X_1\supset \dots\supset p^{-2}X_n
\\&
\supset p^{-1}X_1\supset \dots\supset p^{-1}X_n 
\supset X_1\supset X_2\supset \dots\supset X_n \\&
\supset pX_1 \supset \dots pX_n \supset p^2 X_1 \supset \dots p^2X_n \supset \dots \supset 0
\end{align*}

The right $R$-modules $X_i/X_{i+1}$ and $V_i/V_{i+1}$ are simple. Define the $n^2$ uniserial $R$-modules $U_{i,j}:=V_j/X_i$. Now notice that $$X:=(
\underbrace{{\Q},\dots,{\Q}}_{n},
\underbrace{0,\dots,0}_{n})=\bigcap_{m\ge0}q^mV_j=\sum_{m\ge0}p^{-m}X_i$$ for every $i,j=1,2,\dots,n$. Thus we have a short exact sequence \begin{equation}\label{P}
\xymatrix@1{ 0 \ar[r] &X/X_i \ar[r] &U_{i,j}=V_j/X_i \ar[r] &V_j/X \ar[r] &0.}
\end{equation}

The right $R$-module $X/X_i $ is an artinian uniserial module whose lattice of submodules is order isomorphic to the ordinal number $\omega+1$, where $\omega$ denotes the first infinite ordinal. Since $X/X_i$ is an artinian uniserial $R$-module, its endomorphism ring $\End(X/X_i )$ is a local ring whose maximal ideal consist of all the endomorphisms of $X/X_i$ that are zero on the socle of $X/X_i$.

The right $R$-module $V_j/X$ is a noetherian uniserial module, whose lattice of submodules is order antiisomorphic to the ordinal number $\omega+1$. The cyclic uniserial module $V_j/X$ is annihilated by the two-sided ideal $I:=\left(\begin{smallmatrix} \Lambda_p&0\\
{\bf M}_n({\Q})&0 \end{smallmatrix}\right)$ of $R$, so that it is a module over $R/I\cong\Lambda_q$. In particular, $\End_R(V_j/X)=\End_{\Lambda_q}(V_j/X)$. As a module over $\Lambda_q$, $V_j/X$ turns out to be a cyclic projective uniserial module, hence $V_j/X\cong e\Lambda_q$ for some idempotent $e\in\Lambda_q$. Thus $\End_R(V_j/X)=\End_{\Lambda_q}(V_j/X)\cong\End_{\Lambda_q}(e\Lambda_q)\cong e\Lambda_qe\cong \Z_q$ is a local ring.

Finally, every endomorphism of the $R$-module $U_{i,j}:=V_j/X_i$ maps $X/X_i $ into $X/X_i $, i.e., $X/X_i $ is a fully invariant submodule of $V_j/X_i$, because $X/X_i $ is the Loewy submodule of $V_j/X_i$. Thus the endomorphism ring of the extension
(\ref{P}) in the category $\Cal E$ is canonically isomorphic to the endomorphism ring of the right $R$-module $U_{i,j}$, which is a ring of type $2$.
\end{example}

\begin{example}
Let $R=\Z$ be the ring of integers and let $p\in \Z$ be a prime. Consider the following object of $\Cal U$:
$$
\xymatrix@1{ 0 \ar[r] & \Z/ p\Z \ar[r]^{\iota} & \Z(p^{\infty}) \ar[r]^{\beta} & \Z(p^{\infty}) \ar[r] &0,}
$$
where $\Z(p^{\infty})$ denotes the Pr\"ufer group and $\beta$ is any surjective endomorphism of $\Z(p^{\infty})$ having its kernel equal to $\soc(\Z(p^{\infty}))\cong \Z/ p\Z$. It is easy to check that $I_{B,e,u}\subseteq 
I_{B,m,u}\subseteq I_{B,m,l}=I_{B,e,l}$, and therefore $E_B$ has exactly one maximal right ideal.
\end{example}

\section{Main Theorem}

Let $\Cal U_0$ denote the full subcategory of $\Cal E$ whose objects are of the form
\begin{equation}
\xymatrix@1{ 0 \ar[r] &0 \ar[r] &C' \ar[r] &C \ar[r] &0,}
\end{equation}
for some non-zero uniserial right $R$-module $C$. Similarly, define $\Cal U^0$ to be the full subcategory of $\Cal E$ whose objects are of the form
\begin{equation}
\xymatrix@1{ 0 \ar[r] &A \ar[r] &A' \ar[r] &0 \ar[r] &0,}
\end{equation}
for some non-zero uniserial right $R$-module $A$.

\begin{remark}\label{rema} A short exact sequence \begin{equation}
\xymatrix@1{ 0 \ar[r] &A \ar[r]^\alpha &B \ar[r]^\beta &C \ar[r] &0 } \label{O}
\end{equation} splits if and only if it is the direct sum in the category $\Cal E$ of two short exact sequences \begin{equation}\xymatrix@1{ 0 \ar[r] &A \ar[r]^\gamma &A' \ar[r] &0 \ar[r] &0 }\label{O'}
\end{equation}  and \begin{equation}
\xymatrix@1{ 0 \ar[r] &0 \ar[r] &C' \ar[r] &C \ar[r] &0. }\label{O''}
\end{equation}  

Now suppose we have a split short exact sequence (\ref{O}) with both $A$ and $C$ non-zero modules. The relations between the $(a,b)$-class of (\ref{O}) and the $(a,b)$-classes of (\ref{O'}) and (\ref{O''}) are as follows:

(1) The $(a,b)$-classes of (\ref{O'}) are $[A']_{m,l}\ne [0]_{m,l}$, $[A']_{e,l}\ne [0]_{e,l}$, $[0]_{m,u}$, $ [0]_{e,u}$.

(2) The $(a,b)$-classes of (\ref{O''}) are $[0]_{m,l}$, $ [0]_{e,l}$, $[C']_{m,u}\ne[0]_{m,u}$, $[C']_{e,u}\ne [0]_{e,u}$.

(3) The $(a,b)$-classes of (\ref{O}) are $[B]_{m,l}= [A']_{m,l}$, $[B]_{e,l}= [A']_{e,l}$, $[B]_{m,u}=[C']_{m,u}$, $[B]_{e,u}= [C']_{e,u}$.
\end{remark}

\begin{Lemma}\label{splits} The following conditions are equivalent for an object
 $(\ref{O})$ of $\Cal U$:
\begin{enumerate}
\item[{\rm (a)}]
The short exact sequence $(\ref{O})$ splits.
\item[{\rm (b)}]
There exists an object $(\ref{O'})$ of $\Cal U^0$ such that $[B]_{a,l}=[A']_{a,l}$ for both $a=m,e$.
\item[{\rm (c)}]
There exists an object $(\ref{O''})$ of $\Cal U_0$ such that $[B]_{a,u}=[C']_{a,u}$ for both $a=m,e$.
\end{enumerate}
\end{Lemma}

\begin{proof}
(a)$\Rightarrow$(b) and (a)$\Rightarrow$(c) easily follow from Remark~\ref{rema}.

(b)$\Rightarrow$(a) By hypothesis, there exist $f,g\colon B\rightarrow A'$ such that $f|^A_{A}$ is injective and $g|^A_{A}$ is surjective. From \cite[Lemma 1.4 (a)]{TAMS}, one of $f|^A_{A}$, $g|^A_{A}$ or $(f+g)|^A_{A}$ is an isomorphism. So we can find a morphism $h\colon B\rightarrow A'$ in $\Cal E$ for which $h|^A_{A}$ is an isomorphism. Now it is immediate to see that $(h|^A _{A})^{-1} \gamma^{-1} h \alpha=id_A$, where $\gamma\colon  A\rightarrow A'$ is the isomorphism in the short exact sequence $(\ref{O'})$. This means that (\ref{O}) splits.
Similarly for (c)$\Rightarrow$ (a). 
\end{proof}

\begin{theorem}\label{completo'} Let
\begin{equation}\label{1''x}
\xymatrix@1{ 0 \ar[r] &A_i \ar[r] &B_i \ar[r] &C_i \ar[r] &0,}\qquad i=1,2,\dots,n,
\end{equation}
and
\begin{equation}\label{2x''}
\xymatrix@1{ 0 \ar[r] &A'_j \ar[r] &B'_j \ar[r] &C'_j \ar[r] &0,}\qquad j=1,2,\dots,m,
\end{equation}
be $n+m$ non-zero objects in the category $\Cal E$ with all the modules $A_i,C_i,A'_j,C'_j$ uniserial modules. 
Set $$\begin{array}{l}X_{l}:=\{\,i\mid i=1,2,\dots,n,\ A_i\ne0\,\},\\
X_{u}:=\{\,i\mid i=1,2,\dots,n,\ C_i\ne0\,\},\\
X'_{l}:=\{\,j\mid j=1,2,\dots,m,\ A'_j\ne0\,\},\\
X'_{u}:=\{\,j\mid j=1,2,\dots,m,\ C'_j\ne0\,\}.\end{array}$$
Then $\bigoplus_{i=1}^nB_i\cong\bigoplus_{j=1}^mB'_j$ in the category $\Cal E$ if and only if there exist four bijections $\varphi_{a,b}\colon X_{b}\to X'_{b}$, $a=m,e$ and $b=l,u$, such that $[B_i]_{a,b}=[B'_{\varphi_{a,b}(i)}]_{a,b}$ for every $a=m,e$, $b=l,u$ and $i\in X_{b}$.
\end{theorem}

\begin{proof}
($\Rightarrow$) Suppose $\bigoplus_{i=1}^nB_i\cong\bigoplus_{j=1}^mB'_j$ in the category $\Cal E$. Then $\bigoplus_{i=1}^nA_i\cong\bigoplus_{j=1}^mA'_j$ and $\bigoplus_{i=1}^nC_i\cong\bigoplus_{j=1}^mC'_j$ in $\Mod R$. Taking the Goldie dimensions of these two pairs of isomorphic modules, we get that $|X_l|=|X'_l|$ and $|X_u|=|X'_u|$. Set $r:=|X_l|=|X'_l|$ and $s:=|X_u|=|X'_u|$. Suppose for instance $r\le s$, so that $r+t=s$ for some integer $t\ge0$. The set $\{1,2,\dots,n\}$ is partitioned into three pair-wise disjoint subsets $X_l\cap X_u$, $X_l\setminus X_u$ and $X_u\setminus X_l$, so that $\{1,2,\dots,n\}=(X_l\cap X_u)\dot{\cup}(X_l\setminus X_u)\dot{\cup}(X_u\setminus X_l)$. Then $|X_l\setminus X_u|+t=|X_u\setminus X_l|$. Now set $u:=|X_l\setminus X_u|$, so that $|X_l\cap X_u|=n-2u-t$. Let $i_1,\dots,i_u$ be the elements of $X_l\setminus X_u$ and fix $u$ elements $k_1,\dots,k_u$ in $X_u\setminus X_l$. In the direct sum $\bigoplus_{i=1}^nB_i$, replace the $2u$ sequences 
\begin{equation}\label{11}
\xymatrix@1{ 0 \ar[r] &A_{i_\ell} \ar[r] &B_{i_\ell} \ar[r] &0\ar[r] &0,}\qquad \ell=1,\dots,u,
\end{equation}
and
\begin{equation}\label{12}
\xymatrix@1{ 0 \ar[r] &0 \ar[r] &B_{k_\ell} \ar[r] &C_{k_\ell} \ar[r] &0,}\qquad \ell=1,\dots,u,
\end{equation}
with the $u$ split sequences 
\begin{equation}\label{13}
\xymatrix@1{ 0 \ar[r] &A_{i_\ell} \ar[r] &B_{i_\ell}\oplus B_{k_\ell} \ar[r] &C_{k_\ell}\ar[r] &0,}\qquad \ell=1,\dots,u.
\end{equation} In each of these replacements, the $(a,b)$-classes $[B_{i_\ell}]_{m,l}\ne [0]_{m,l}$, $[B_{i_\ell}]_{e,l}\ne [0]_{e,l}$, $[0]_{m,u}$, $ [0]_{e,u}$ of (\ref{11}) and 
$[0]_{m,l}$, $ [0]_{e,l}$, $[B_{k_\ell}]_{m,u}\ne[0]_{m,u}$, $[B_{k_\ell}]_{e,u}\ne [0]_{e,u}$  of (\ref{12}) are replaced with the
 $(a,b)$-classes $[B_{i_\ell}\oplus B_{k_\ell}]_{m,l}= [B_{i_\ell}]_{m,l}$, $[B_{i_\ell}\oplus B_{k_\ell}]_{e,l}= [B_{i_\ell}]_{e,l}$, $[B_{i_\ell}\oplus B_{k_\ell}]_{m,u}=[B_{k_\ell}]_{m,u}$, $[B_{i_\ell}\oplus B_{k_\ell}]_{e,u}= [B_{k_\ell}]_{e,u}$  of (\ref{13}). That is, only some zero $(a,b)$-classes are deleted in the replacement.
 
Now fix a simple right $R$-module $S$ and replace each of the remaining $t$ sequences indexed in $(X_u\setminus X_l)\setminus\{k_1,\dots,k_u\}$ with the direct sum of the sequence and the sequence $\xymatrix@1{ 0 \ar[r] &S \ar[r] &S\ar[r] &0\ar[r] &0}$. After these substitutions, the object $(\bigoplus_{i=1}^nB_i)\oplus S^t$ of $\Cal E$ is the direct sum of the $n-2u-t$ sequences (\ref{1''x}) indexed in $X_l\cap X_u$ and $u+t$ split sequences. All these $n-u$ summands are objects of $\Cal U$.

Similarly, the object $(\bigoplus_{j=1}^mB'_j)\oplus S^t$ of $\Cal E$ can be written as the direct sum of $m-2u'-t$ sequences indexed in $X'_l\cap X'_u$, plus $u'$ split sequences of the form
 \begin{equation}
\xymatrix@1{ 0 \ar[r] &A'_{j_v} \ar[r] &B'_{j_v}\oplus B'_{j'_v} \ar[r] &C'_{j'_v}\ar[r] &0,}\qquad v=1,\dots,u',
\end{equation} 
plus $t$ split sequences of the form  $\xymatrix@1{ 0 \ar[r] &S \ar[r] &B'_j \oplus S\ar[r] &C'_j \ar[r] &0}$.

Moreover, $\bigoplus_{i=1}^nB_i\cong\bigoplus_{j=1}^mB'_j$ implies that $(\bigoplus_{i=1}^nB_i)\oplus S^t\cong(\bigoplus_{j=1}^mB'_j)\oplus S^t$ in $\Cal E$, so that we can apply Proposition~\ref{completo} to this isomorphism. Thus, fix $a=m,e$ and any short exact sequence
\begin{equation}
\xymatrix@1{ 0 \ar[r] &U \ar[r] &V \ar[r] &W\ar[r] &0}\label{V}
\end{equation} in $\Cal U$. Then the number of indices $i\in X_l\cap X_u$ with $[B_i]_{a,l}=[V]_{a,l}$ plus the number of indices $\ell=1,\dots,u$ with $[B_{i_\ell}\oplus B_{k_\ell}]_{a,l}=[V]_{a,l}$ plus the number of indices $i\in (X_u\setminus X_l)\setminus\{k_1,\dots,k_u\}$ such that $[B_i\oplus S]_{a,l}=[V]_{a,l}$ is equal to the number of indices $j\in X'_l\cap X'_u$ with $[B'_j]_{a,l}=[V]_{a,l}$ plus the number of indices $v=1,\dots,u'$ with $[B'_{j_v}\oplus B'_{k'_v}]_{a,l}=[V]_{a,l}$ plus the number of indices $j\in (X'_u\setminus X'_l)\setminus\{k'_1,\dots,k'_{u'}\}$ such that $[B'_j\oplus S]_{a,l}=[V]_{a,l}$. From Remark~\ref{rema}, it follows that the number of indices $i\in X_l\cap X_u$ with $[B_i]_{a,l}=[V]_{a,l}$ plus the number of indices $\ell=1,\dots,u$ with $[B_{i_\ell}]_{a,l}=[V]_{a,l}$ plus the number of indices $i\in (X_u\setminus X_l)\setminus\{k_1,\dots,k_u\}$ such that $[S]_{a,l}=[V]_{a,l}$ (that is, $t$ if $[S]_{a,l}=[V]_{a,l}$ and $0$ if $[S]_{a,l}\ne [V]_{a,l}$) is equal to the number of indices $j\in X'_l\cap X'_u$ with $[B'_j]_{a,l}=[V]_{a,l}$ plus the number of indices $v=1,\dots,u'$ with $[B'_{j_v}]_{a,l}=[V]_{a,l}$ plus the number of indices $j\in (X'_u\setminus X'_l)\setminus\{k'_1,\dots,k'_{u'}\}$ with $[S]_{a,l}=[V]_{a,l}$ (that is, $t$ if $[S]_{a,l}=[V]_{a,l}$ and $0$ if $[S]_{a,l}\ne [V]_{a,l}$).

This is true for every sequence (\ref{V}) in $\Cal U$, so that, by Remark~\ref{rema}, it is also true for every sequence 
(\ref{V}) in $\Cal U^0$. Notice that, for every sequence 
(\ref{V}) in $\Cal U$ or $\Cal U^0$, one has that $[B'_j]_{a,l}=[V]_{a,l}$ implies $A'_j\ne 0$. Therefore, the number of indices $i=1,2,\dots,n$ with $[B_i]_{a,l}=[V]_{a,l}$ is equal to the number of indices $j=1,2,\dots,m$ with $[B'_j]_{a,l}=[V]_{a,l}$.

Similarly for the case $r\ge s$ and the $(a,b)$-classes with $b=u$. 

($\Leftarrow$) Suppose there exist four bijections $\varphi_{a,b}\colon X_{b}\to X'_{b}$, $a=m,e$ and $b=l,u$, such that $[B_i]_{a,b}=[B'_{\varphi_{a,b}(i)}]_{a,b}$ for every $a=m,e$, $b=l,u$ and $i\in X_{b}$. Set $r:=|X_l|=|X'_l|$ and $s:=|X_u|=|X'_u|$. Suppose $r\le s$, so that $r+t=s$ for some integer $t\ge0$. The set $\{1,2,\dots,n\}$ is partitioned into three pair-wise disjoint subsets $X_l\cap X_u$, $X_l\setminus X_u$ and $X_u\setminus X_l$, so that $\{1,2,\dots,n\}=(X_l\cap X_u)\dot{\cup}(X_l\setminus X_u)\dot{\cup}(X_u\setminus X_l)$. Then $|X_l\setminus X_u|+t=|X_u\setminus X_l|$. Now set $u:=|X_l\setminus X_u|$, so that $|X_l\cap X_u|=n-2u-t$. Let $i_1,\dots,i_u$ be the elements of $X_l\setminus X_u$ and fix $u$ elements $k_1,\dots,k_u$ in $X_u\setminus X_l$. Fix a simple right $R$-module $S$ and consider the sequence \begin{equation}\xymatrix@1{ 0 \ar[r] &S \ar[r] &S\ar[r] &0\ar[r] &0.}\label{00}\end{equation} Let $i_1,\dots,i_u$ be the elements of $X_l\setminus X_u$ and fix $u$ elements $k_1,\dots,k_u$ in $X_u\setminus X_l$. Like in the proof of the previous implication, we have that, of the $n$ sequences $(\ref{1''x})$, $n-2u-t$ are in $\Cal U$, $u$ are in $\Cal U_0$, and $u+t$ are in $\Cal U^0$. Replace the $n$ sequences $(\ref{1''x})$ with $n-u$ sequences, which are:

(1) the $n-2u-t$ sequences $(\ref{1''x})$ that are in $\Cal U$;

(2) $u$ sequences that are the direct sum of the sequences $(\ref{1''x})$ in $\Cal U_0$ plus $u$ of the sequences $(\ref{1''x})$  in $\Cal U^0$;

(3) $t$  sequences that are the direct sum of the remaining $t$ sequences $(\ref{1''x})$  in $\Cal U^0$ plus the sequence (\ref{00}).

\noindent These new $n-u$ sequences are indexed in $X_u$ and are in $\Cal U$. 

The $(a,b)$-classes of the new $n-u$ sequences with their multiplicity are the same as the non-zero $(a,b)$-classes of the old $n$ sequences $(\ref{1''x})$ for $b=u$, and are  the same as the non-zero $(a,b)$-classes of the old $n$ sequences $(\ref{1''x})$ plus
the $(a,b)$-class $[S]_{a,b}$ with multiplicity $t$ for
for $b=l$.

Similarly for the $m$ sequences (\ref{2x''}). We can replace them with $m-u'$ sequences, which are:

(1) the $m-2u'-t$ sequences $(\ref{2x''})$ that are in $\Cal U$;

(2) $u'$ sequences that are the direct sum of the sequences $(\ref{2x''})$ in $\Cal U_0$ plus $u'$ of the sequences $(\ref{2x''})$ in $\Cal U^0$;

(3) $t$  sequences that are the direct sum of the remaining $t$ sequences $(\ref{2x''})$  in $\Cal U^0$ plus the sequence (\ref{00}).

\noindent These $m-u'$ sequences are indexed in $X'_u$ and are all in $\Cal U$. 
The $(a,b)$-classes of these new $m-u'$ sequences with their multiplicity are the same as the non-zero $(a,b)$-classes of the old $m$ sequences $(\ref{2x''})$ for $b=u$, and are  the same as the non-zero $(a,b)$-classes of the old $m$ sequences $(\ref{2x''})$ plus
the $(a,b)$-class $[S]_{a,b}$ with multiplicity $t$ for
for $b=l$.

The existence of the four bijections $\varphi_{a,b}\colon X_{b}\to X'_{b}$, $a=m,e$ and $b=l,u$ with $[B_i]_{a,b}=[B'_{\varphi_{a,b}(i)}]_{a,b}$ for every $a=m,e$, $b=l,u$ and $i\in X_{b}$ implies that $|X_u |=|X'_u|$, that is, $n-u=m-u'$, and that there are four bijections between the new $n-u$ sequences and the new $m-u'$ sequences that preserve the $(a,b)$-classes. Since all the new sequences are in $\Cal U$, we can apply Proposition~\ref{completo} to the new $n-u+m-u'$ sequences, getting that $\bigoplus_{i=1}^nB_i\oplus S^t\cong\bigoplus_{j=1}^mB'_j\oplus S^t$ in the category $\Cal E$. Now the sequence $S$ cancels from direct sums in $\Cal E$ because its endomorphism ring is a division ring \cite[Theorem 4.5]{libro}. (Notice that Evans' result is stated in \cite[Theorem 4.5]{libro} only for modules, but its proof shows that it holds in any additive category. Moreover, the endomorphism ring of the sequence (\ref{00}) in the category $\Cal E$ is the endomorphism ring of the simple module $S$, which is a division ring.) Thus 
$\bigoplus_{i=1}^nB_i\cong\bigoplus_{j=1}^mB'_j$ in $\Cal E$. Similarly for the case $r\ge s$, making use of the exact sequence $\xymatrix@1{ 0 \ar[r] &0 \ar[r] &S\ar[r] &S\ar[r] &0}$. 
\end{proof}

The authors are very grateful to the referee of the paper, who discovered a mistake hidden in the proof of one of the implications of Proposition~\ref{completo}.

\bibliographystyle{amsalpha}

\end{document}